\documentclass[10pt,a4paper]{amsart}
\usepackage[utf8]{inputenc}
\usepackage{amsmath}
\usepackage{amsfonts}
\usepackage{amssymb}
\usepackage{graphicx}
\usepackage{lipsum}

\usepackage[toc,page]{appendix}

\newtheorem{theorem}{Theorem}[section]
\newtheorem{lemma}[theorem]{Lemma}
\newtheorem{proposition}[theorem]{Proposition}

\newtheorem{definition}[theorem]{Definition}

\newtheorem{example}{Example}
\newtheorem{remark}{Remark}

\begin{document}

\author{G. I. Montecinos}
\title{Analytic solutions for the Burgers
  equation with source terms}

\address{Center for Mathematical Modeling (CMM) \\
 Universidad de  Chile\\
 Beauchef 851, Edificio Norte,  Piso 7, Santiago - Chile}

\email{G. I. Montecinos:  gmontecinos@dim.uchile.cl}

\maketitle

\begin{abstract}

Analytic solutions for Burgers equations with source terms, possibly stiff,  represent an important element to assess
numerical schemes.  Here we present a procedure, based on the characteristic technique to obtain analytic solutions for these equations with smooth initial conditions.

\end{abstract}

\section{Introduction}
Exact solutions of hyperbolic balance laws are very useful to assess
the performance of numerical schemes. Non-linearity of
partial differential equation, as well as, stiffness of source terms are
desirable features to be recovered by numerical schemes.  
The analytic solutions presented here contain these elements.

Exact solutions for Burgers equations are generally obtained by separation of
variables \cite{Estevez:2002a}, regularization techniques \cite{Norgard:2008a},
expansion methods \cite{Wang:2008a}, to mention but a few.  Here, the strategy to solve these equations is based on the characteristic
curve method, see \cite{thomas:1995a,leVeque:1992a}. The solution is obtained in two steps. 
First, the solution of an Ordinary Differential
Equation (ODE) called here, {\it equivalent ODE}, is
obtained. This ODE  is constructed by following the conventional characteristic curve method and contains the influence of the source term.  Second, the solution of an
 ODE, called here  {\it characteristic ODE} is obtained. This equation is defined  in the $x-t$ plane and the solution of the {\it equivalent ODE} is included.  In this way the influence of the source term is
present in the definition of the characteristic curve. 

To ensure the solvableness of the {\it equivalent ODE}, 
the existence of a primitive function for the reciprocal of the source
term is required. Subsequently, if the solution of the {\it equivalent ODE} has a primitive function then the {\it characteristic ODE} is solvable. If these requirements are satisfied, analytic solutions can be obtained. However, a non-linear algebraic equation has to be solved in
the general case, which require the initial condition to be a
continuous function.

This work is organized as follows. In section \ref{general:source} the
procedure is presented. In section \ref{no:source}, the procedure is applied to
the homogeneous Burgers equation. In section \ref{linear:source}  the
solution for a linear source term is obtained. In section \ref{quadratic:source}
the analytic solution for  a quadratic source term is
obtained.  In section  \ref{non-linear:source}, the methodology is applied to solve Burgers's equation with a more general non-linear source term.  Finally, in section \ref{conclusion} the main results of this work are summarized.

\section{Analytic solutions for Burgers equations with a special class of source
  terms}\label{general:source}
 To start, let us consider a balance law in the form
\begin{eqnarray}\label{eq:exact:0-0}
\left.
\begin{array}{c}
\partial_t q(x,t)+\partial_x\biggl(\frac{q(x,t)^2}{2}\biggr) = s(q(x,t))\;,\\
q(x,0) = h_0(x)\;, 
\end{array}
\right\}      
\end{eqnarray}
where $h_0(x)$ is a continuous  initial condition and $s(q)$ is the source
term which has to satisfy the properties of  lemma \ref{lemma:1} shown below. 

By following the characteristic curve
method we define a curve $x$ in the  $x-t$ plane, which satisfies
\begin{eqnarray}\label{eq:exact:0-1}
\left.
\begin{array}{c}
\frac{d }{dt}x(t) = q(x(t),t)  \;,\\
x(0) = y\;,
\end{array}
\right\}
\end{eqnarray}
with $y$ a constant value. 
This ODE will receive the name of {\it characteristic ODE}. 
On the other hand, we define  
\begin{eqnarray}
  \hat{q}(t) = q(x(t),t)\;, h(0):=q(x(0),0)=h_0(y)\;,
\end{eqnarray}
where $y$ is that given in (\ref{eq:exact:0-1}). With these definitions, (\ref{eq:exact:0-0}) becomes an ODE
given by
\begin{eqnarray}\label{eq:exact:0-2}
\left.
\begin{array}{c}
\frac{d}{dt}\hat{q}(t) = s(\hat{q}(t))   \;,\\
\hat{q}(0) = h(0)\;.
\end{array}
\right\}
\end{eqnarray}
This ODE is called here, {\it equivalent ODE}.

Now, let us see the following definition in order to call for
ensuring the existence of solutions.
\begin{definition}
The primitive of a function $f(x)$ is a differentiable  function $F(x)$ such that
\begin{eqnarray}
\frac{d }{dx } F(x) = f(x)\;.
\end{eqnarray}
\end{definition}
Next lemma sets the type of source terms which are considered in this work.
\begin{lemma}\label{lemma:1}
  If $s(\hat{q})^{-1}$ contains a primitive with respect to  $\hat{q}$, then  the {\it
    equivalent ODE} is solvable. Additionally, there exists a function
  $\mathcal{E}(t,h(0))$ such that 
\begin{eqnarray}\label{eq:exact:0-3}
\left.
\begin{array}{c}
\frac{\partial }{\partial t}\mathcal{E}(t,h(0)) = s(\hat{q}(t))\;,\\
\mathcal{E}(0,h(0)) = h(0)\;\\
\end{array}
\right\}
\end{eqnarray}
and $\hat{q}(t) = \mathcal{E}(t,h(0))\;.$
\end{lemma}
\begin{proof}
We integrate (\ref{eq:exact:0-1}) as follows
\begin{eqnarray}\label{eq:exact:0-4}
  \int_{h(0)}^{\hat{q}(t)} s(q)^{-1} dq =\int_0^tdt  \;,
\end{eqnarray}
as $s(q)^{-1}$ has a primitive function, there exists $G(q)$ such that 
\begin{eqnarray}
  \frac{d}{dq}G(q) = s(q)^{-1}\;. 
\end{eqnarray}
Therefore, $\mathcal{E}(t,h(0)) $ is the solution to $G(\mathcal{E})-G(h(0)) -t=0  \;.$ 
So, if there exists the inverse function of  $G(q)$, which is  denoted here
by $G^{-1}(q)$, the function $\mathcal{E}$ is explicitly given by
\begin{eqnarray}
 \mathcal{E}(t,h(0))= G^{-1}(t+G(h(0)))  \;.
\end{eqnarray}
In any case  the exact solution is obtained as

\begin{eqnarray}
\hat{q}(t)= \mathcal{E}(t,h(0))  \;.
\end{eqnarray}
\end{proof}

Once the solution to the {\it equivalent  ODE} is
available and observing that $h(0):=h_0(y)$, the {\it characteristic ODE} takes the form
\begin{eqnarray}\label{eq:exact:0-5}
\left.
 \begin{array}{c}
  \frac{d}{dt}x(t)=\mathcal{E}(t,h_0(y))\;,\\
x(0) = y\;.
\end{array}
\right\}
\end{eqnarray}
The following lemma deals with the existence of solution for this ODE.
\begin{lemma}\label{lemma:2} 
If $\mathcal{E}(t,h_0(y)$ has a primitive function $\mathcal{F}(t,h_0(y))$,
such that \\$\frac{d}{dt}\mathcal{F}(t,h_0(y))=\mathcal{E}(t,h_0(y))\;$
and $\mathcal{F}(0,h_0(y))=0\;.$ Then, the
characteristic ODE has the exact solution
\begin{eqnarray}\label{eq:exact:0-6}
x = y +\mathcal{F}(t,h_0(y))\;.  
\end{eqnarray}
\end{lemma}
\begin{proof}
Integrating the ODE (\ref{eq:exact:0-5}), we obtain
\begin{eqnarray}
 \begin{array}{c}
 \int_{y}^{x} dx=\int_0^t\mathcal{E}(t,h_0(y))dt\;.
\end{array}
\end{eqnarray}
So, by using the properties of $\mathcal{F}(t,h_0(y))$, the result holds.
\end{proof}

\begin{remark}
The value $y$, is  a constant for the characteristic ODE
(\ref{eq:exact:0-5}). However, if values $x$ and $t$ are set in (\ref{eq:exact:0-6}), there exist a
constant $y$ satisfying (\ref{eq:exact:0-6}). Therefore we can
identify a such constant by $y=y(x,t)\;.$ 
\end{remark}

\begin{proposition}\label{proposition:1}
If $s(q)^{-1} $ and $\mathcal{E}(t,h_0(y))$ have their respective
primitive functions. The problem (\ref{eq:exact:0-0}) has the exact solution 
\begin{eqnarray}\label{burger:exact-1}
q(x,t) = \mathcal{E}(t,h_0(y))\;,  
\end{eqnarray}
where $y$ satisfies 
\begin{eqnarray}\label{burger:exact-2}
x = y +\mathcal{F}(t,h_0(y))\;,  
\end{eqnarray}
with $\mathcal{F}(t,h_0(y))$ the primitive of $\mathcal{E}(t,h_0(y))$
with respect to $t\;.$
\end{proposition}

\begin{proof}
The construction of this function is given by lemmas (\ref{lemma:1}) and
(\ref{lemma:2}).
Now, we are going to probe that $q(x,t)$ solves (\ref{eq:exact:0-0}).

By the chain rule
\begin{eqnarray}
\left.
  \begin{array}{ccc}
    \partial_t q &=& \frac{\partial }{\partial y}
    \mathcal{E}\frac{\partial y}{\partial t} + \frac{\partial }{\partial t}
    \mathcal{E}\;,\\

    \partial_x q &=& \frac{\partial }{\partial y}
    \mathcal{E}\frac{\partial y}{\partial x}\;.
  \end{array}
\right\}
\end{eqnarray}
Then 
\begin{eqnarray}
  \begin{array}{c}
    \partial_t q +q\partial_x q = \frac{\partial }{\partial y}
    \mathcal{E}(\frac{\partial y}{\partial t}+q\frac{\partial y}{\partial x}) + \frac{\partial }{\partial t}
    \mathcal{E}\;.\\
  \end{array}
\end{eqnarray}
On the other hand, we have
\begin{eqnarray}
  \begin{array}{ccc}
   0 &=&  \frac{\partial y}{\partial t} + \frac{\partial }{\partial t} \mathcal{F}+\frac{\partial }{\partial
   (h_0)}\mathcal{F}h_0(y)'\frac{\partial y}{\partial t}\;,\\

   1 &=&  \frac{\partial y}{\partial x} + \frac{\partial }{\partial
   (h_0)}\mathcal{F}h_0(y)'\frac{\partial y}{\partial x}\;.\\
  \end{array}
\end{eqnarray}
Therefore, as $\mathcal{F}$ is the primitive of $\mathcal{E}$ and from (\ref{burger:exact-1}), we have
\begin{eqnarray}
  \begin{array}{c}
\frac{\partial y}{\partial t}+q\frac{\partial y}{\partial x}=
\biggl( 1+\frac{\partial }{\partial h_0}F h_0' \biggr)^{-1}
(-\frac{\partial }{\partial t} \mathcal{F} + q) =0 \;.

  \end{array}
\end{eqnarray}
Finally, we note that 
\begin{eqnarray}
  \begin{array}{c}
\frac{\partial }{\partial t} \mathcal{E}(t,h_0(y)) = s(q(x,t))\;
  \end{array}
\end{eqnarray}
and so, the result holds.
\end{proof}

\begin{remark}
Note that in the general case, (\ref{burger:exact-2})  is an algebraic equation and the bisection method is used to solve them. The regularity requirements of the bisection method is that, functions have to be at least continuous ones, which is ensured by taking $h_0(x)$ a continuous function. 
\end{remark}

In the following sections we will present some applications of the strategy describe above.
\section{Homogeneous Burgers equations}\label{no:source}
In this section we apply the methodology seen in section
\ref{general:source}, to solve the problem
\begin{eqnarray}\label{eq:no:0-0}
\left.
\begin{array}{c}
\partial_t q(x,t) +\partial_x \biggl(\frac{q(x,t)^2}{2} \biggr) = 0\;,  \\
q(x,0) = h_0(x)\;.
\end{array}
\right\}
\end{eqnarray}
Then, the {\it equivalent ODE} has the form
\begin{eqnarray}\label{eq:no:0-1}
\left.
\begin{array}{c}
\frac{d}{dt}\hat{q}(t) = 0\;,\\
\hat{q}(0) = h(0)\;.  
\end{array}
\right\}
\end{eqnarray}
This ODE is solvable and the exact solution has the form
$\hat{q}(t)=\mathcal{E}(t,h(0))=h(0)\;.$  Therefore, the  {\it
  characteristic ODE} takes the form
\begin{eqnarray}\label{eq:no:0-2}
\left.
\begin{array}{c}
\frac{d}{dt}x(t) = h_0(y)\;,\\
x(0) = y\;,  
\end{array}
\right\}
\end{eqnarray}
as $\mathcal{E}(t,h(0))$ has the primitive function
$\mathcal{F}(t,h(0))=t  h(0)\;, $ we obtain 

\begin{eqnarray}\label{eq:no:0-3}
x = y + th_0(y)\;.  
\end{eqnarray}
However, $h_0(y)=h_0(x(0))=q(x(0),0)=q(x,t)$. Therefore,
(\ref{eq:no:0-3}) becomes 
\begin{eqnarray}\label{eq:no:0-4}
x = y + tq(x,t)\;.  
\end{eqnarray}
Hence the exact solution  to (\ref{eq:no:0-0}), is computed as
\begin{eqnarray}\label{eq:no:0-4}
q(x,t) = h_0(x-tq(x,t))\;,  
\end{eqnarray}
which is the conventional solution of the homogeneous Burgers equation for general
smooth initial conditions.

\begin{example}
Let us consider the problem (\ref{eq:no:0-0}) with $h_0(x) = x\;.$  A simple inspection shows that, $u(x,t)=\frac{x}{t+1}$, is the exact solution.   Following the present methodology we are going to obtain this exact solution. 

To start, let us note that  {\it equivalent ODE} has the exact solution 
\begin{eqnarray}\label{eq:no:ex:0-0}
\hat{q}(t) = h(0)\;.  
\end{eqnarray}
Therefore, the exact solution  to (\ref{eq:no:0-0}) is given by
\begin{eqnarray}\label{eq:no:ex:0-2}
q(x,t) = h_0(y)\;,  
\end{eqnarray}
with $y$ satisfying the relationship obtained from the {\it
  characteristic ODE}, which by using the form of $h_0(y)$ has the form
\begin{eqnarray}\label{eq:no:ex:0-1}
x = y  + th_0(y)=y(1+t)\;.  
\end{eqnarray}
So, by solving for $y$  the sought solution is obtained
\begin{eqnarray}\label{eq:no:ex:0-3}
q(x,t) = h_0(y)=y=\frac{x}{t+1}\;.  
\end{eqnarray}
\end{example}

\section{Burgers's equation with a linear-source term}\label{linear:source}
Now, let us consider the PDE given by
\begin{eqnarray}\label{eq:linear:0-0}
\left.
\begin{array}{c}
\partial_t q(x,t) +\partial_x \biggl(\frac{q(x,t)^2}{2} \biggr) =
\beta q(x,t)\;,  \\
q(x,0) = h_0(x)\;.
\end{array}
\right\}
\end{eqnarray}
Then, the {\it equivalent ODE} has the form
\begin{eqnarray}\label{eq:linear:0-1}
\left.
\begin{array}{c}
\frac{d}{dt}\hat{q}(t) = \beta \hat{q}(t)\;,\\
\hat{q}(0) = h(0)\;.  
\end{array}
\right\}
\end{eqnarray}
This ODE is solvable and the exact solution has the form
$\hat{q}(t)=\mathcal{E}(t,h(0))=h(0)exp(\beta t)\;.$  Therefore, the  {\it
  characteristic ODE} takes the form
\begin{eqnarray}\label{eq:linear:0-2}
\left.
\begin{array}{c}
\frac{d}{dt}x(t) = h_0(y)e^{\beta t}\;,\\
x(0) = y\;,  
\end{array}
\right\}
\end{eqnarray}
as $\mathcal{E}(t,h(0))$ has the primitive function
$\mathcal{F}(t,h(0))= (\frac{e^{\beta t}-1}{\beta}) h(0)\;, $ we obtain 

\begin{eqnarray}\label{eq:linear:0-3}
x = y + \frac{e^{\beta t}-1}{\beta}h_0(y)\;.  
\end{eqnarray}
Therefore, the exact solution to (\ref{eq:linear:0-0}) is given by
\begin{eqnarray}\label{eq:linear:0-4}
q(x,t) = h_0(y)e^{\beta t}\;,  
\end{eqnarray}
with $y$ satisfying (\ref{eq:linear:0-3}).

\begin{example}
  The function $u(x,t)=\frac{\beta x e^{\beta t}}{\beta - 2 + e^{\beta t}}$ is the exact solution
of (\ref{eq:linear:0-0}) with $h_0(x) = \frac{\beta x}{\beta -1}\;,$
with $\beta \neq 1$\;.

To find this solution by using the present methodology we
note that the {\it equivalent ODE} has the exact solution 
\begin{eqnarray}\label{eq:linear:ex:0-0}
\hat{q}(t) = h(0) e^{\beta t}\;.  
\end{eqnarray}
On the other hand, the  {\it characteristic ODE} due to the form of $h_0(y)$, has the exact solution
\begin{eqnarray}\label{eq:linear:ex:0-1}
x = y  + h_0(y)\frac{e^{\beta t}-1}{\beta }= y (1 +\frac{e^{\beta t -1
  }}{\beta -1} )\;.  
\end{eqnarray}
Therefore, the exact solution  to (\ref{eq:linear:0-0}) is given by
\begin{eqnarray}\label{eq:linear:ex:0-2}
q(x,t) = h_0(y)e^{\beta t}\;,  
\end{eqnarray}
with $y$ satisfying (\ref{eq:linear:ex:0-1}), which finally provides
\begin{eqnarray}\label{eq:linear:ex:0-3}
q(x,t) = h_0(y)e^{\beta t}=\frac{\beta x e^{\beta t}}{\beta -2 +e^{\beta t}}\;.  
\end{eqnarray}
\end{example}

\begin{example}
Let us consider the problem (\ref{eq:linear:0-0}) in the interval $[0,1]\;,$ with the
initial condition $h_0(x)\;,$ given by
\begin{eqnarray}\label{eq:linear:ex-3:1}
h_0(x) =\left\{
\begin{array}{ccc}
2 &,& x \leq 0.1 \;,\\
2+\frac{(x-0.1)}{0.1} &,& 0.1 \leq x \leq 0.2 \;,\\
3 &,& 0.2 \leq x \leq 0.4 \;,\\
3-\frac{(x-0.4)}{0.2} &,& 0.4 \leq x \leq 0.6 \;,\\
2 &,& x \geq 0.6 \;.\\
 \end{array}
\right.
\end{eqnarray}
In that case, the equation (\ref{eq:linear:0-3}) is not explicitly solvable for $y$. In this work we use the bisection method
to obtain $y$. 

\begin{figure}
  \centering
  \includegraphics[scale=0.5]{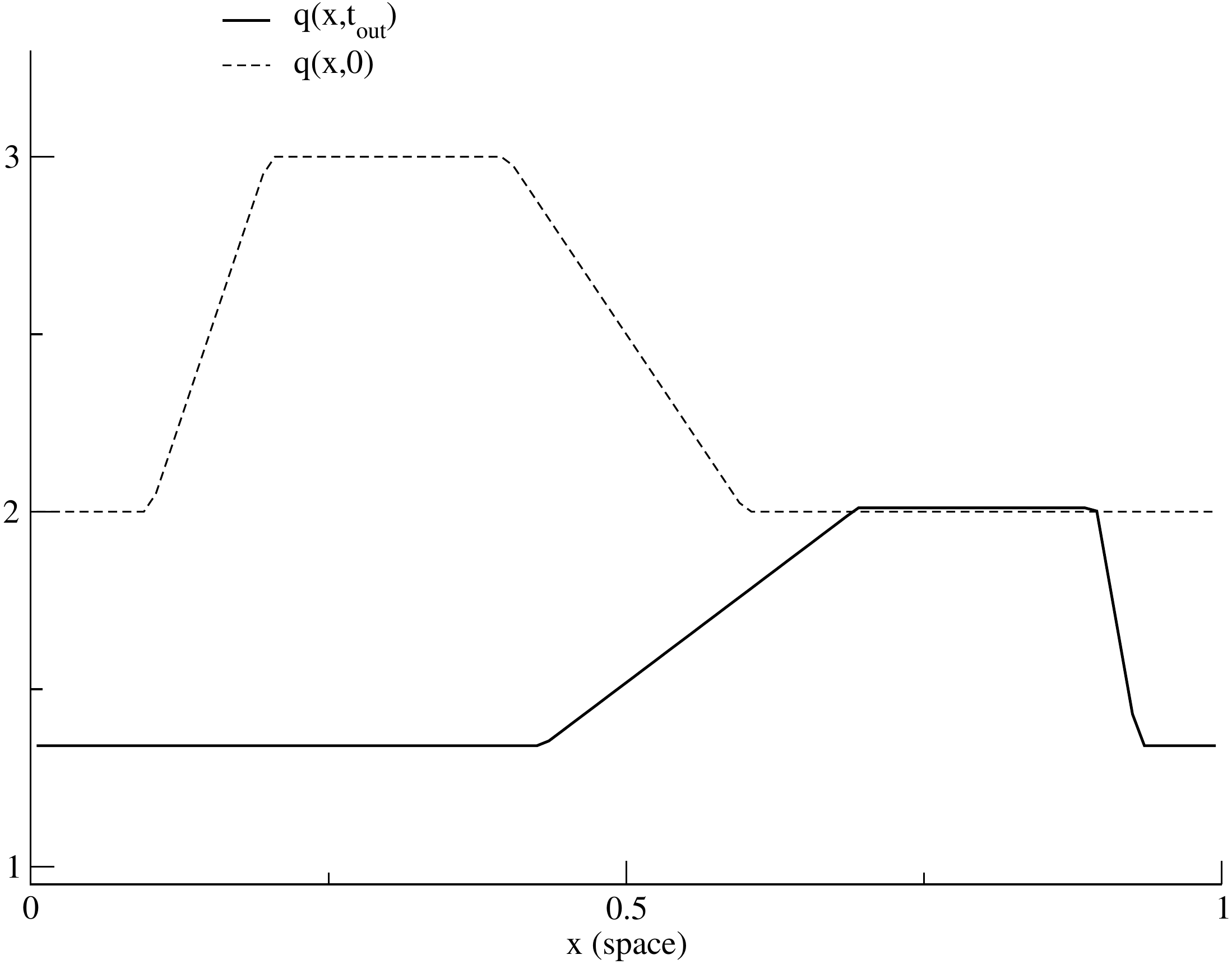}
\caption{Burgers's equation with linear source term. 
Initial condition (dash line) and the solution at time $t_{out}=0.2$
(full line) for $\beta = -2$ .}\label{fig:linSource:1}
\end{figure}

Figure \ref{fig:linSource:1}, shows initial condition
(\ref{eq:linear:ex-3:1}) and the solution of (\ref{eq:linear:0-0})
at time $t_{out}=0.2$ for $\beta=-2\;.$

\end{example}

\section{Burgers's equation with a quadratic source term}\label{quadratic:source}
Let us consider the partial differential equation
\begin{eqnarray}\label{eq:exact:1-1}
\left.
\begin{array}{c}
\partial_t q(x,t)+\partial_x\biggl(\frac{q(x,t)^2}{2}\biggr) = \beta q(x,t)^2\;,\\
q(x,t) = h_0(x)\;.
\end{array}
\right\}
\end{eqnarray}\label{eq:exact:2-0}
The {\it equivalent ODE} has the form
\begin{eqnarray}\label{eq:exact:3-0}
\left.
\begin{array}{c}
\frac{d \hat{q}(t)}{dt}=\beta \hat{q}(t)^2\;,
\\
\hat{q}(0) = h(0)\;,
\end{array}
\right\}
\end{eqnarray}
which is solvable and the exact solution is
\begin{eqnarray}\label{eq:exact:3-1}
\hat{q}(t) = \mathcal{E}(t,h(0))=\frac{h(0)}{1-\beta t h(0)}\;.
\end{eqnarray}
On the other hand, we note that the {\it characteristic ODE} 
\begin{eqnarray}\label{eq:exact:2-2}
\left.
  \begin{array}{c}
\frac{d}{dt}x(t) = q(x,t)\;,\\
x(0)=y\;,
  \end{array}
\right\}  
\end{eqnarray}
for $q(x,t)$ satisfying (\ref{eq:exact:3-1}) has the solution
\begin{eqnarray}\label{eq:exact:3-2}
x = y -\frac{ln(1-\beta t h_0(y))}{\beta} \;.
\end{eqnarray}
Therefore, the solution of (\ref{eq:exact:1-1}) is given by
\begin{eqnarray}\label{eq:exact:4-0}
q(x,t) = \frac{h_0(y)}{1-\beta t h_0(y)}\;,
\end{eqnarray}
with $y$ satisfying (\ref{eq:exact:3-2}).

\begin{example}
Let us consider (\ref{eq:exact:1-1}) in the interval $[0,1]\;,$ with initial condition 
\begin{eqnarray}
h_0(x) = \sin(2\pi x)\;.
\end{eqnarray}
With this choice, $y$ is not explicitly obtained from
(\ref{eq:exact:3-2}). Therefore we use the bisection method. Figure
\ref{fig:quadratic:1}, shows the initial condition (dash line ) and
the  solution at $t_{out}=0.15$ and for $\beta = -2\;.$
\begin{figure}
  \centering
  \includegraphics[scale=0.5]{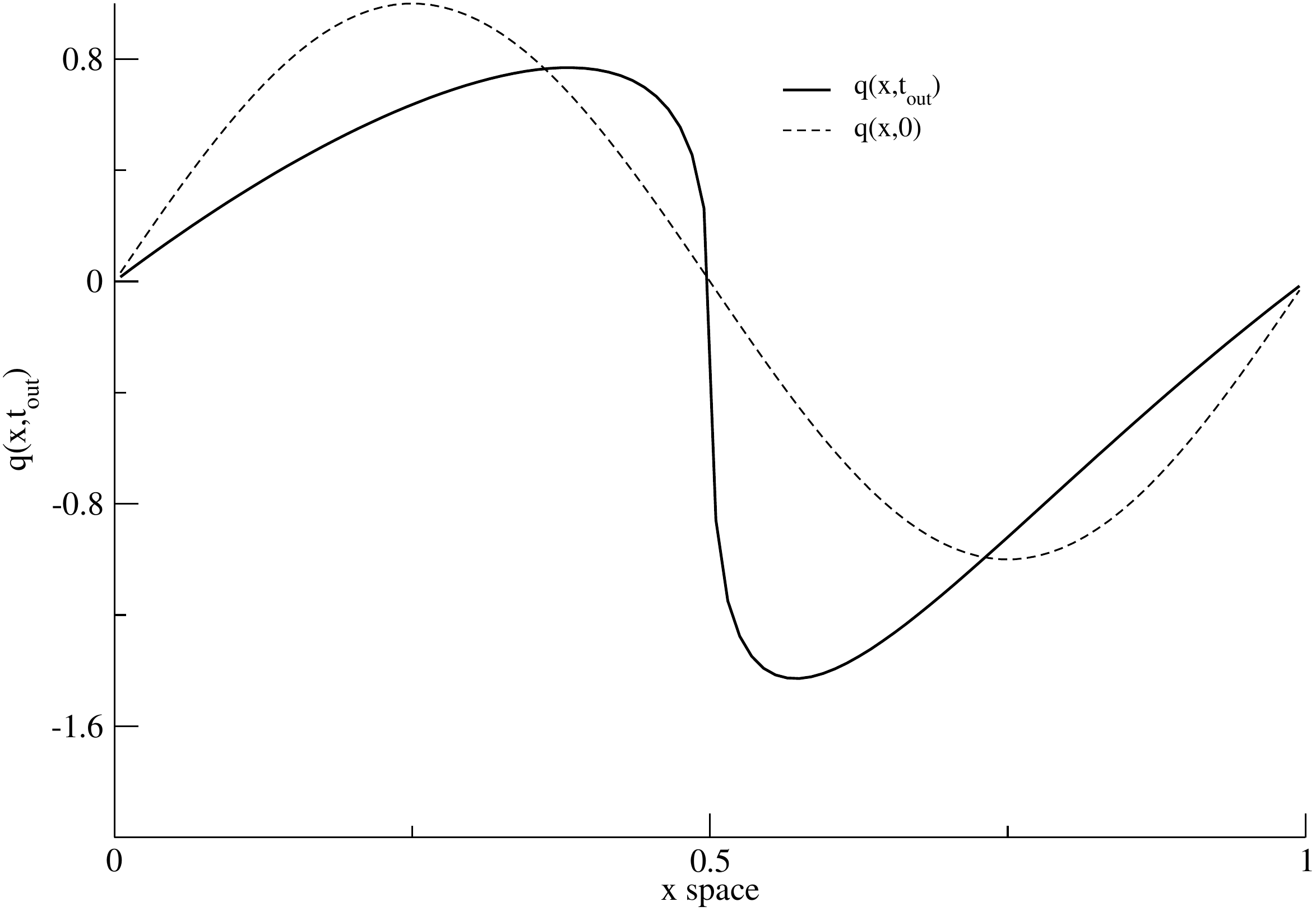}
\caption{Burgers's equation with quadratic source term. Initial
  condition (dash line) and the exact solution at time $t_{out}=
  0.15$ and for $\beta = -2\;.$}\label{fig:quadratic:1} 
\end{figure}
\end{example}

\section{Burgers's equation with a non-linear source term}\label{non-linear:source}
Let us consider the partial differential equation
\begin{eqnarray}\label{eq:exact-non-lin:1-1}
\left.
\begin{array}{c}
\partial_t q(x,t)+\partial_x\biggl(\frac{q(x,t)^2}{2}\biggr) = e^{\beta q(x,t)}\;,\\
q(x,t) = h_0(x)\;.
\end{array}
\right\}
\end{eqnarray}
So, the {\it equivalent ODE} 
\begin{eqnarray}\label{eq:exact-non-lin:3-0}
\left.
\begin{array}{c}
\frac{d \hat{q}(t)}{dt}=e^{\beta \hat{q}(t)}\;,
\\
\hat{q}(0) = h(0)\;,
\end{array}
\right\}
\end{eqnarray}
has the exact solution
\begin{eqnarray}
\hat{q}(t) =\mathcal{E}(t,h(0))= -\frac{ln(e^{-\beta h(0)}-\beta t  ) }{\beta}\;.
\end{eqnarray}
On the other hand, the {\it characteristic ODE} 
\begin{eqnarray}\label{eq:exact-non-lin:2-2}
\left.
  \begin{array}{c}
\frac{d}{dt}x(t) = -\frac{ln(e^{-\beta h(0)}-\beta t  ) }{\beta}\;,\\
x(0)=y\;.
  \end{array}
\right\}  
\end{eqnarray}
with $\hat{q}(t)=q(x(t),t)$ satisfying (\ref{eq:exact-non-lin:2-2})
has the exact solution
\begin{eqnarray}\label{eq:exact-non-lin:3-2}
x = y +\frac{ (e^{-\beta h_0(y)}-\beta t  )ln(e^{-\beta h_0(y)}-\beta t  )
  +\beta t+\beta h_0(y)e^{-\beta h_0(y)}}{\beta^2} 
 \;.
\end{eqnarray}
Therefore, the solution to (\ref{eq:exact-non-lin:1-1}) is given by
\begin{eqnarray}
q(x,t)=-\frac{ln(e^{-\beta h_0(y)}-\beta t  ) }{\beta}\;,
\end{eqnarray}
with $y$ a solution of (\ref{eq:exact-non-lin:3-2}).

\begin{example}
  Let us consider (\ref{eq:exact-non-lin:1-1}) in the interval
  $[0,1]$, with the initial condition
  \begin{eqnarray}
    \label{eq:1}
    h_0(x) = 2\biggl(\frac{1-w(x)}{2}\biggr)+\biggl(\frac{1+w(x)}{2}\biggr)\;.
  \end{eqnarray}
with $w(x) =\displaystyle \frac{0.3-x}{\sqrt{(0.3-x)^2+\varepsilon}}\;.$
Figure \ref{figure:non-lin:1} shows the initial condition for
$\varepsilon = 10^{-4}\;$ and the respective solution at time $t_{out}=0.26$ for $\beta=-1$.

\begin{figure}
  \centering
  \includegraphics[scale=0.4]{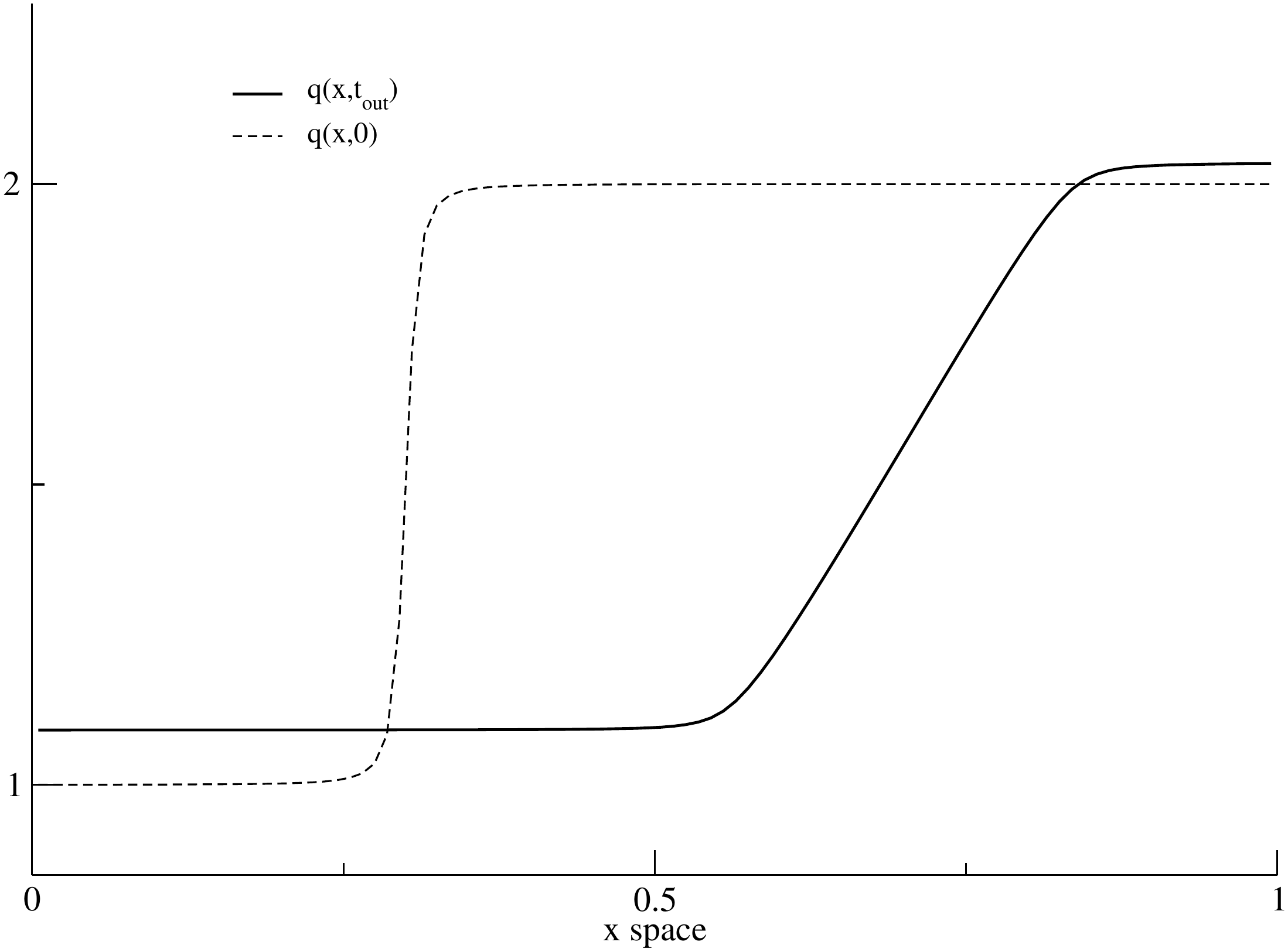}
\caption{Bueger's equation with exponential source term. Initial
  condition (dash line ) and solution at $t_{out}=0.26$ (full
  line). Parameters $\beta=-1\;$ and $\varepsilon = 10^{-4}\;.$}\label{figure:non-lin:1}
\end{figure}

\end{example}
\section{Conclusions}\label{conclusion}
In this work  we have presented a procedure to obtain analytic
solutions to the Burgers equation with a family of source terms aimed to
provide a set of tests which is suitable for the assessment of numerical schemes. 
The family of source terms is that formed by functions containing  primitive
functions with respect to its arguments.  
By following the characteristic method, we have defined two ODE's,
which were called {\it equivalent ODE} and {\it characteristic ODE}.
This family of source terms allows us the solvableness of
these ODE's and as a consequence the sought solution of the original
PDE is found. 
In the solution of the {\it characteristic ODE} an algebraic equation,
which depends on the initial condition, is obtained. The requirement of
continuity for the initial condition allows us to find the root of
this algebraic equation through the bisection method.

\bibliographystyle{plain} 
\bibliography{ref}

\begin{thebibliography}{1}

\bibitem{Estevez:2002a}
P.~G. Est\'{e}vez, C.~Qu, and S.~Zhang.
\newblock Separation of variables of a generalized porous medium equation with
  nonlinear source.
\newblock {\em Journal of Mathematical Analysis and Applications}, 275(1):44 --
  59, 2002.

\bibitem{leVeque:1992a}
R.J. LeVeque.
\newblock {\em Numerical Methods for Conservation Laws}.
\newblock Lectures in Mathematics ETH Z{\"u}rich, Department of Mathematics
  Research Institute of Mathematics. Springer, 1992.

\bibitem{Norgard:2008a}
G.~Norgard and K.~Mohseni.
\newblock {A regularization of the Burgers equation using a filtered convective
  velocity}.
\newblock {\em Journal of Physics A: Mathematical and Theoretical},
  41(344016):21 pp, 2008.

\bibitem{thomas:1995a}
J.W. Thomas.
\newblock {\em Numerical Partial Differential Equations: Finite Difference
  Methods}.
\newblock Number v. 1 in Graduate Texts in Mathematics. Springer, 1995.

\bibitem{Wang:2008a}
M.~Wang, X.~Li, and J.~Zhang.
\newblock {The (G'/G)-expansion method and travelling wave solutions of
  nonlinear evolution equations in mathematical physics}.
\newblock {\em Physics Letters A}, 372(4):417 -- 423, 2008.

\end{thebibliography}
\end{document}